\documentclass[12pt]{amsart}
\usepackage{amsthm}

\usepackage{amssymb}
\usepackage{verbatim}
\usepackage[curve,matrix,arrow]{xy}
\usepackage{amsmath,amsfonts}
\usepackage{graphicx}
\usepackage{epsf}

\usepackage{graphics}

\newcommand{\mathsym}[1]{{}}

\newcommand{\thmref}[1]{Theorem~\ref{#1}}

\newcommand{\lemref}[1]{Lemma~\ref{#1}}
\newcommand{\eqnref}[1]{Equation~(\ref{#1})}
\newcommand{\remref}[1]{Remark~\ref{#1}}

\newcommand{\figref}[1]{Figure~\ref{#1}}

  {\end{list}}

\def\li{L_{i}}
\def\ri{R_{i}}

\def\RR{{\mathbb R}}

\newtheorem{theorem}{Theorem}[section]

\newtheorem{lemma}[theorem]{Lemma}

\theoremstyle{example}

\newtheorem{remark}[theorem]{Remark}

\theoremstyle{definition}

\theoremstyle{notation}

\newcommand{\dd}[1]{\delta_{#1}}

\newcommand{\hj}[3]{\hat{j}_{#1}(#2,#3)}

\newcommand{\ga}{\Gamma}

\newcommand{\tg}{\tau(\Gamma)}

\newcommand{\ee}[1]{E(#1)}
\newcommand{\vv}[1]{V(#1)}
\newcommand{\va}{\upsilon}

\newcommand{\pp}{p_{i}}

\newcommand{\qq}{q_{i}}

\def\can{{\mathop{\rm can}}}

\def\BDV{{\mathop{\rm BDV}}}

\def\<{\langle }
\def\>{\rangle }
\newcommand{\secref}[1]{\S\ref{#1}}

\def\cd{c_{\mu_{D}}}

\def\elg{\ell (\ga)}

\def\diag{\text{diag}}
\newcommand{\am}{\mathrm{A}}

\newcommand{\dm}{\mathrm{D}}

\newcommand{\lm}{\mathrm{L}}
\newcommand{\plm}{\mathrm{L^+}}

\newcommand{\zm}{\mathrm{Z}}
\newcommand{\rmm}{\mathrm{R}}

\begin{document}

\title[Computation of Arakelov-Green Functions]
{Explicit Computation of Certain Arakelov-Green Functions}

%AMS classification
%Primary
%MSC2010: Arithmetic varieties and schemes; Arakelov theory; heights
%MSC2010: None of the above, but in this section
%Secondary
%MSC2010: Applications of graph theory
%MSC2010: Programming involving graphs or networks

\author{Zubeyir Cinkir}
\address{Zubeyir Cinkir\\
Zirve University\\
%Faculty of Education\\
Department of Mathematics Education\\
Gaziantep\\
TURKEY.}
\email{zubeyirc@gmail.com}


\begin{abstract}
Arakelov-Green functions defined on metrized graphs have important role in relating arithmetical problems on algebraic curves into graph theoretical problems. In this paper, we clarify the combinatorial interpretation of certain Arakelov-Green functions by using electric circuit theory. The formulas we gave clearly show that such functions are piece-wisely defined, and each piece is a linear or quadratic function on each pair of edges of metrized graphs. These formulas lead to an algorithm for explicit computation of Arakelov-Green functions.
\end{abstract}

\maketitle

\section{Introduction}\label{section introduction}
%\vskip .1 in

Algebraic geometers have powerful tools due to intersection theory over complex numbers to study curves and varieties in general.
Observing the success of algebraic geometers, it is the desire of number theorist and arithmetic geometers to utilize the intersection theory for studying arithmetic properties of algebraic curves. However, if one works over fields other than complex numbers, many difficulties arise for various nice properties of complex numbers are no longer in use. Additional new tools should be used to overcome these difficulties.
This is what S. Arakelov did over archimedean fields in his studies which we know as Arakelov theory by now \cite{A}. Arakelov introduced an intersection pairing on arithmetic surfaces. The key part was to consider the contribution to the intersection number that comes from the infinite places. This contribution is defined by using Arakelov-Green function for the Riemann surfaces associated to the arithmetic surfaces. He used analysis and studied Laplace operator on those associated Riemann surfaces to derive global results on arithmetic surfaces. We note that the use of admissible metrised line bundles, metrized line bundles satisfying certain analytic criteria, on arithmetic surfaces is another important tool considered in Arakelov theory. G. Faltings' arithmetic analogues of Riemann-Roch theorem and adjunction formula from classical intersection theory on surfaces are two striking examples for the successes of Arakelov theory. These kinds of successes enabled Faltings to prove Mordell conjecture \cite{F} among other results in arithmetic geometry.

We have a similar story for non-archimedean fields. In this case, we have metrized graphs as non-archimedean analogues of Riemann surfaces. Again we have Arakelov-Green functions and Laplacian operators on metrized graphs. Reduction graphs, the dual graphs associated to the special fibre curve, are examples of metrized graphs. R. Rumely, who introduced metrized graphs to study arithmetic properties of algebraic curves and developed capacity theory \cite{RumelyBook}, contributed to the development of local intersection theory for algebraic curves defined over non-archimedean fields. Metrized graphs were further developed by T. Chinburg and Rumely in \cite{CR} and by S. Zhang in \cite{Zh1}. Rumely and T. Chinburg introduced "capacity pairing" and used metrized graphs in their work \cite{CR}. Later, S.
Zhang introduced another intersection pairing as a non-archimedean analogue of Arakelov's pairing on a Riemann surface, and he showed that analogous Riemann-Roch theorem and adjunction formula hold for this admissible pairing \cite{Zh1}. In \cite{BRh}, M. Baker and R. Rumely used harmanic analysis on metrized graphs to study Arakelov-Green functions and related continuous Laplacian operators. Various arithmetic results are obtained after these studies. For example, the proof of Effective Bogomolov Conjecture over function fields of characteristic zero \cite{C7}, \cite{Zh2}.

Metrized graphs and Arakelov-Green functions on metrized graphs have important roles in the articles  \cite{BRh}, \cite{CR}, \cite{C7}, \cite{RumelyBook}, \cite{Zh1} and \cite{Zh2}.
%
%Faltings and Elkies gave lower bounds for $g_{\mu}(x, y)$-discriminant sums (see [La,
%§VI, Theorem5.1]). Estimates for these sums can then be used to prove the effectivity
%of certain Arakelov divisors on an arithmetic surface (see [La, §V, Theorem 4.1]).
%Another related work concerning Arakelov-Green functions is to give lower bounds for $g_{\mu}(x,y)$-discriminant sums over a metrized graph. %Baker and Rumely in \cite{BRh} obtained Elkies-type bounds for $g_{\mu}(x, y)$-discriminant sums generalizing those found for the circle %graph in [HS].
%
%[HS] M Hindry and J. Silverman, On Lehmer's conjecture for elliptic curves. In: S´eminaire de Th´eorie
%des Nombres. Progr.Math. 91. Birkh¨auser Boston, Boston,MA, 1990, pp. 103–116,.
%
The basic interest about Arakelov-Green functions is to find their values on any given points of metrized graphs. 
Our aim in this article is to address this issue by finding an efficient algorithm that can be used for both symbolic and numerical
computations of Arakelov-Green functions.

In \secref{sec review}, we give a short description of metrized graphs and their discrete Laplacian matrix.
In \secref{sec resistance function}, we first review basic facts about the resistance function $r(x,y)$ on a metrized graph $\ga$.
Then we obtain formulas that express $r(x,y)$ in terms of the end points of the edges that contain $x$ and $y$ (see \lemref{lem resistance one edge}, \lemref{lem resistance px} and \thmref{thm resistance pxy}). This means that one needs
basically the effective resistance values between any two vertices in $\ga$ to obtain the values of $r(x,y)$.

In \secref{sec arakelov green function}, we first describe Arakelov-Green function on a metrized graph $\ga$. 
Baker and Rumely showed that Arakelov-Green function $g_{\mu_{can}}(x,y)$ can be expressed in terms of the tau constant $\tg$ of the metrized graph $\ga$ and the resistance function $r(x,y)$ (see \thmref{thm arakelov-green canonical}). Combining this fact and our results from \secref{sec resistance function} about resistance function, we obtain our main result in \thmref{thm main1}. In this way, we show that $g_{\mu_{can}}(x,y)$ on $\ga$ is a piece-wisely defined quadratic or linear function in both $x$ and $y$ by explicitly giving the coefficients of each piece in terms of the effective resistance values between the related vertices of $\ga$. If $x$ (or $y$) belongs to an edge whose removal disconnects $\ga$, $g_{\mu_{can}}(x,y)$ is linear in $x$ (or $y$). Otherwise it will be quadratic. We suggest that a matrix $\zm$ of size $e \times e$ can be used to describe $g_{\mu_{can}}(x,y)$, where $e$ is the number of edges in $\ga$.

In \secref{sec computations}, we give several examples of computations of $g_{\mu_{can}}(x,y)$ by finding the matrix $\zm$ which we call the value matrix. We know that the tau constant can be computed symbolically and numerically by using either theoretical work in various cases (\cite{C4} and \cite{C5})  or computer algorithms in all cases \cite{C3}. Therefore, we conclude the same things for computation of $g_{\mu_{can}}(x,y)$ by both using the results of \secref{sec arakelov green function} and our previous results on the tau constant.
%Since $g_{\mu_{can}}(x,y)=-\frac{1}{2}r(x,y)+\tg$ as shown by Baker and Rumely, 

\section{Metrized Graphs}\label{sec review}
%\vskip .1 in

In this section, we give a brief review of metrized graphs and their discrete Laplacian matrix.

A metrized graph $\ga$ is a finite connected graph equipped with a distinguished parametrization of each of its edges.
A metrized graph $\ga$ can have multiple edges and self-loops.
For any given $p \in \ga$,
the number $\va(p)$ of directions emanating from $p$ will be called the \textit{valence} of $p$.
%, and will be denoted by $\va(p)$.
By definition, there can be only finitely many $p \in \ga$ with $\va(p)\not=2$.

For a metrized graph $\ga$, we will denote a vertex set for $\ga$ by $\vv{\ga}$.
We require that $\vv{\ga}$ be finite and non-empty and that $p \in \vv{\ga}$ for each $p \in \ga$ if $\va(p)\not=2$. For a given metrized graph $\ga$, it is possible to enlarge the
vertex set $\vv{\ga}$ by considering additional valence $2$ points as vertices.

For a given metrized graph $\ga$ with vertex set $\vv{\ga}$, the set of edges of $\ga$ is the set of closed line segments with end points in $\vv{\ga}$. We will denote the set of edges of $\ga$ by $\ee{\ga}$. However, if
$e_i$ is an edge, by $\ga-e_i$ we mean the graph obtained by deleting the {\em interior} of $e_i$.

%We define the genus of $\ga$ to be the first Betti number $g(\ga):=e-v+1$ of the graph $\ga$, where $e$ and $v$ are the number of edges and %vertices of $\ga$, respectively.

We denote the length of an edge $e_i \in \ee{\ga}$ by $\li$, which represents a positive real number. The total length of $\ga$, which is denoted by $\elg$, is given by $\elg=\sum_{i=1}^e\li$.

To have a well-defined discrete Laplacian matrix $\lm$ for a metrized
graph $\ga$, we first choose a vertex set $\vv{\ga}$ for $\ga$ in
such a way that there are no self-loops, and no multiple edges
connecting any two vertices. This can be done by
enlarging the vertex set by considering additional valence two points as vertices
whenever needed. We call such a vertex set $\vv{\ga}$ \textit{adequate}.
If distinct vertices $p$ and $q$ are the end points of an edge, we
call them \textit{adjacent} vertices.

Let $\ga$ be a metrized graph with $e$ edges and an adequate vertex set
$\vv{\ga}$ containing $v$ vertices. Fix an ordering of the vertices
in $\vv{\ga}$. Let $\{L_1, L_2, \cdots, L_e\}$ be a labeling of the
edge lengths. The matrix $\am=(a_{pq})_{v \times v}$ given by
\[
a_{pq}=\begin{cases} 0, & \quad \text{if $p = q$, or $p$ and $q$ are
not adjacent}.\\
\frac{1}{L_k}, & \quad \text{if $p \not= q$, and an edge of length $L_k$ connects $p$ and $q$.}\\
%$p$ and $q$ are connected by} \text{ an edge of length $L_k$}\\
\end{cases}
\]
is called the \textit{adjacency matrix} of $\ga$. Let $\dm=\diag(d_{pp})$ be
the $v \times v$ diagonal matrix given by $d_{pp}=\sum_{s \in
\vv{\ga}}a_{ps}$. Then $\lm:=\dm-\am$ is called the \textit{discrete
Laplacian matrix} of $\ga$. That is, $ \lm =(l_{pq})_{v \times v}$ where
\[
l_{pq}=\begin{cases} 0, & \; \, \text{if $p \not= q$, and $p$ and $q$
are not adjacent}.\\
-\frac{1}{L_k}, & \; \, \text{if $p \not= q$, and $p$ and $q$ are
connected by} \text{ an edge of length $L_k$}\\
-\sum_{s \in \vv{\ga}-\{p\}}l_{ps}, & \; \, \text{if $p=q$}
\end{cases}.
\]

One can find more information about $\lm$ in \cite[Section 3]{C3} and the references therein.

\section{Resistance Function $r(x,y)$}\label{sec resistance function}
%\vskip .1 in

In this section, we study the resistance and the voltage functions on a metrized graph $\ga$.
After reviewing the facts that we will use about these functions, we consider the following problem.
If one considers these functions on a graph, having only combinatorial nature, consisting of vertices and edges between these vertices,
one can compute the resistance and the voltage functions by using the discrete Laplacian matrix of the graph.
However, these functions are continuous functions on a metrized graph $\ga$. A metrized graph being more than a combinatorial graph has additional structures, but still have the combinatorial properties of a graph. Therefore, there should be way to
relate the values of continuous resistance and voltage functions on $\ga$ with the values of discrete resistance and voltage functions
on the vertices of a combinatorial graph. Our goal is to clarify this relation in this section. The results we obtain in this section will be used in the next section.

For any $x$, $y$, $z$ in $\ga$, the voltage function $j_z(x,y)$ on a metrized graph
$\ga$ is a symmetric function in $x$ and $y$, which satisfies
$j_x(x,y)=0$ and $j_z(x,y) \geq 0$ for all $x$, $y$, $z$ in $\ga$.
For each vertex set $\vv{\ga}$, $j_{z}(x,y)$ is
continuous on $\ga$ as a function of all three variables.
%(see also \cite[Exercise 9]{BF}).
For fixed $z$ and $y$ it
has the following physical interpretation: If $\Gamma$ is viewed
as a resistive electric circuit with terminals at $z$ and $y$,
with the resistance in each edge given by its length, then
$j_{z}(x,y)$ is the voltage difference between $x$ and $z$,
when unit current enters at $y$ and exits at $z$ (with reference
voltage $0$ at $z$).

The effective resistance between two points $x, \, y$ of a metrized graph $\ga$ is given by $r(x,y)=j_y(x,x),$
where $r(x,y)$ is the resistance function on $\ga$. The resistance function inherits certain properties of the voltage function.
For any $x$, $y$ in $\ga$,  $r(x,y)$ on
$\ga$ is a symmetric function in $x$ and $y$, and it satisfies
$r(x,x)=0$. For each vertex set $\vv{\ga}$, $r(x,y)$ is
continuous on $\ga$ as a function of two variables and
%(see also \cite[Page 12]{BF}).
%As the physical interpretation suggests,
$r(x,y) \geq 0$ for all $x$, $y$ in $\ga$.
If a metrized graph $\Gamma$ is viewed as a
resistive electric circuit with terminals at $x$ and $y$, with the
resistance in each edge given by its length, then $r(x,y)$ is
the effective resistance between $x$ and $y$ when unit current enters
at $y$ and exits at $x$.

The proofs of the facts mentioned above can be found in \cite{CR}, \cite[sec 1.5 and sec 6]{BRh}, and \cite[Appendix]{Zh1}.
The voltage function $j_{z}(x,y)$ and the resistance function $r(x,y)$ are also studied in the articles \cite{BF} and \cite{C1}. 
%\cite{C6}.

We will denote by $R_i$ the resistance between the end points of an edge $e_i$ of a graph $\ga$ when the interior of the edge $e_i$ is deleted from $\ga$.

Let $\ga$ be a metrized graph with $p \in \vv{\ga}$, and let $e_i \in \ee{\ga}$ having end points $\pp$ and $\qq$.
If $\ga -e_i$ is connected, then $\ga$
can be transformed to the graph in Figure \ref{fig 2termp}
by circuit reductions. More details on this fact can be found in the articles \cite{CR} and \cite[Section 2]{C2}.
%where $d=\li$, $a=R_{a_i,p}$,
%$b=R_{b_i,p}$ and $c=R_{c_i,p}$.
Note that in \figref{fig 2termp}, we have $R_{a_i,p} = \hj{\pp}{p}{\qq}$,
$R_{b_i,p} = \hj{\qq}{p}{\pp}$, $R_{c_i,p} = \hj{p}{\pp}{\qq}$, where $\hj{x}{y}{z}$
is the voltage function in $\ga-e_i$. We have $R_{a_i,p}+R_{b_i,p}=\ri$ for each $p \in \ga$.

If $\ga-e_i$ is not connected, we set $R_{b_i,p}=\ri=\infty$ and $R_{a_i,p}=0$ if $p$ belongs to the component of $\ga-e_i$
containing $\pp$, and we set $R_{a_i,p}=\ri=\infty$ and $R_{b_i,p}=0$ if $p$ belongs to the component of $\ga-e_i$
containing $\qq$. We will use this notation for the rest of the paper.
\begin{figure}
\centering
\includegraphics[scale=1]{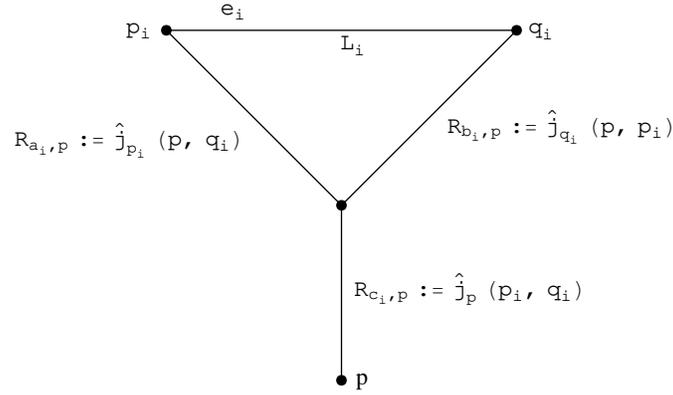} \caption{Circuit reduction with reference to an edge and a point.} \label{fig 2termp}
\end{figure}

Recall that the function $r(x,y)$ is defined on $\ga$ and has nonnegative real number values. Therefore, when we write an equality as in \lemref{lem resistance one edge} below, we mean that $x, \, y \in \ga$ on the left side of equality and that $x, \, y$ are the corresponding real numbers via the parametrization. For example, if $x$ is on edge $e_i$ of length $\li$ with end points $\pp$ and $\qq$, then we consider a parametrization identifying $e_i$ by the interval $[0,\li]$ so that the points $\pp$ and $\qq$ correspond to $0$ and $\li$, respectively; and that $x \in [0,\li]$. We follow this approach in the rest of the paper. One should note that the direction of parametrization makes no problem in our computations as long as one is careful about the adjustment of the relevant formulas.
\begin{lemma}\label{lem resistance one edge}
Let $e_i \in \ee{\ga}$ be an edge of length $\li$ with end points $p_i$ and $q_i$. If both $x$ and $y$ belong to the same edge $e_i$, then
$$r(x,y)= |x-y|-(x-y)^2 \frac{\li - r(\pp,\qq)}{\li^2}.$$
\end{lemma}
\begin{proof}
Using circuit reductions, this case can be illustrated as in \figref{fig sameedgexy}.
With abuse of notation, $x$ and $y$ denote both points on $e_i$ and their distances to the vertex $\pp$.
\begin{figure}
\centering
\includegraphics[scale=1]{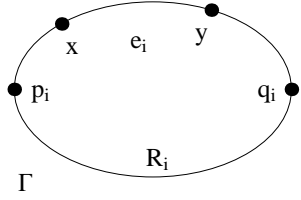} \caption{Circuit reduction with reference to an edge $e_i$ having end points $\pp$ and $\qq$.} \label{fig sameedgexy}
\end{figure}
The result follows from the fact that $r(\pp,\qq)=\frac{\li \ri}{\li+\ri}$ and that $x$ and $y$ are connected by two parallel edges with edge lengths $|x-y|$ and $\li+\ri+|x-y|$.
\end{proof}
Note that $\li$ and $r(\pp,\qq)$ can be expressed in terms of the entries of the discrete Laplacian matrix $\lm$ and its pseudo inverse $\plm$, respectively.
In this way, whenever $x$ and $y$ are chosen from the same edge, we can express the continuous function $r(x,y)$ as a piecewise linear or quadratic function with coefficients obtained by using the discrete graph representation of metrized graphs.
The condition that both $x$ and $y$ are on the same edge is an essential hypothesis in \lemref{lem resistance one edge}.
A relevant question is that what would be the corresponding formula of $r(x,y)$ if $x$ and $y$ are chosen from different edges of $\ga$.
In the rest of this section, we provide an answer to this question. First, we need the following technical lemma:
\begin{lemma}\label{lem resistance px}
Let $e_i \in \ee{\ga}$ be an edge of length $\li$ with end points $p_i$ and $q_i$. If $x$ belongs to the edge $e_i$, for any vertex $p \in \vv{\ga}$ we have
$$r(p,x)= -x^2 \frac{\li - r(\pp,\qq)}{\li^2}+x \frac{\li-r(\pp,\qq)+r(p,\qq)-r(p,\pp)}{\li}+r(p,\pp).$$
\end{lemma}
\begin{proof}
Using circuit reductions, this case can be illustrated as in \figref{fig xonedgeei}.
\begin{figure}
\centering
\includegraphics[scale=1.5]{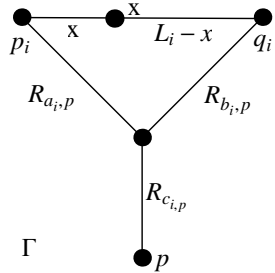} \caption{Circuit reduction with reference to an edge and a vertex.} \label{fig xonedgeei}
\end{figure}
Applying circuit reductions on the electric circuit given in \figref{fig xonedgeei}, we obtain
%\begin{equation}\label{eqn resistances}
%\begin{split}
%r(p,x) &= \frac{(x+R_{a_i,p})(\li -x+R_{b_i,p})}{\li+R_i}+R_{c_i,p} \\
%r(p,\pp) &= \frac{R_{a_i,p}(\li +R_{b_i,p})}{\li+R_i}+R_{c_i,p} \\
%r(p,\qq) &= \frac{R_{b_i,p}(\li +R_{a_i,p})}{\li+R_i}+R_{c_i,p} \\
%r(\pp,\qq)&=\frac{\li R_i}{\li+R_i}.
%\end{split}
%\end{equation}
\begin{align}\label{eqn resistances}
r(p,x)&= \frac{(x+R_{a_i,p})(\li -x+R_{b_i,p})}{\li+R_i}+R_{c_i,p},& \qquad r(\pp,\qq)&=\frac{\li R_i}{\li+R_i}, \\
r(p,\pp)&= \frac{R_{a_i,p}(\li +R_{b_i,p})}{\li+R_i}+R_{c_i,p},&  \qquad r(p,\qq) &= \frac{R_{b_i,p}(\li +R_{a_i,p})}{\li+R_i}+R_{c_i,p}.
\end{align}
Then the result follows from these equations.
%TODO $(1)$ and $(2)$. (\ref{eqn resistances}).
\end{proof}

\begin{theorem}\label{thm resistance pxy}
Let $e_i \in \ee{\ga}$ be an edge of length $\li$ with end points $p_i$ and $q_i$, and let $e_j \in \ee{\ga}$ be an edge of length $L_j$ with end points $p_j$ and $q_j$. Suppose the edges $e_i$ and $e_j$ are distinct, but their end points are not necessarily distinct. If $x$ belongs to the edge $e_i$ and $y$ belongs to the edge $e_j$, we have
\begin{equation*}\label{eqn resistances pxy}
\begin{split}
r(x,y) &=  -x^2 \frac{\li - r(\pp,\qq)}{\li^2}-y^2 \frac{L_j - r(p_j,q_j)}{L_j^2}
 +\frac{2xy}{\li L_j} \big(j_{p_j}(\pp,q_j)-j_{p_j}(q_i,q_j)\big)\\
& \qquad +\frac{x}{\li} \big(\li-2j_{p_i}(\qq,p_j)\big)
 +\frac{y}{L_j} \big(L_j-2j_{p_j}(\pp,q_j)\big)
 +r(\pp,p_j).
\end{split}
\end{equation*}
%\begin{equation*}\label{eqn resistances pxy}
%\begin{split}
%r(x,y) &=  -x^2 \frac{\li - r(\pp,\qq)}{\li^2}-y^2 \frac{L_j - r(p_j,q_j)}{L_j^2}\\
%&\qquad +\frac{xy}{\li L_j} \big(r(\pp,p_j)-r(\pp,q_j)-r(\qq,p_j)+r(\qq,q_j)\big)\\
%& \qquad +\frac{x}{\li} \big(\li-r(\pp,\qq)+r(\qq,p_j)-r(\pp,p_j)\big)\\
%&\qquad +\frac{y}{L_j} \big(L_j-r(p_j,q_j)+r(\pp,q_j)-r(\pp,p_j)\big)\\
%& \qquad +r(\pp,p_j).
%\end{split}
%\end{equation*}
\end{theorem}
\begin{proof}
Applying \lemref{lem resistance px} with edge $e_j$ containing $y$ and vertex $\pp$, we obtain
\begin{equation}\label{eqn rpiy}
r(\pp,y)= -y^2 \frac{L_j - r(p_j,q_j)}{L_j^2}+y \frac{L_j-r(p_j,q_j)+r(\pp,q_j)-r(\pp,p_j)}{L_j}+r(\pp,p_j).
\end{equation}
Similarly, applying \lemref{lem resistance px} with edge $e_j$ containing $y$ and vertex $\qq$ gives
\begin{equation}\label{eqn rqiy}
r(\qq,y)= -y^2 \frac{L_j - r(p_j,q_j)}{L_j^2}+y \frac{L_j-r(p_j,q_j)+r(\qq,q_j)-r(\qq,p_j)}{L_j}+r(\qq,p_j).
\end{equation}
Now, we fix a point $y \in \ee{e_j}$ and consider it as a vertex, and apply \lemref{lem resistance px} with edge $e_i$ containing $x$ and vertex $y$. In this way, we obtain
\begin{equation}\label{eqn rxy}
r(x,y)= -x^2 \frac{\li - r(\pp,\qq)}{\li^2}+x \frac{\li-r(\pp,\qq)+r(y,\qq)-r(y,\pp)}{\li}+r(y,\pp).
\end{equation}
Using the fact that resistance function is symmetric, we substitute Equations (\ref{eqn rpiy}) and (\ref{eqn rqiy}) into \eqnref{eqn rxy} to obtain
\begin{equation}\label{eqn resistances pxy2}
\begin{split}
r(x,y) &=  -x^2 \frac{\li - r(\pp,\qq)}{\li^2}-y^2 \frac{L_j - r(p_j,q_j)}{L_j^2}\\
&\qquad +\frac{xy}{\li L_j} \big(r(\pp,p_j)-r(\pp,q_j)-r(\qq,p_j)+r(\qq,q_j)\big)\\
& \qquad +\frac{x}{\li} \big(\li-r(\pp,\qq)+r(\qq,p_j)-r(\pp,p_j)\big)\\
&\qquad +\frac{y}{L_j} \big(L_j-r(p_j,q_j)+r(\pp,q_j)-r(\pp,p_j)\big)\\
& \qquad +r(\pp,p_j).
\end{split}
\end{equation}
Then the result follows using the fact that $2j_x(y,z)=r(x,y)+r(x,z)-r(y,z)$ for any $x, \, y, \, z \in \ga$.
\end{proof}

\begin{remark}\label{rem resistance xy on bridge}
Whenever the edges $e_i$ and $e_j$ that $x$ and $y$ belongs to are bridges, i.e. $\ga-e_i$ or $\ga-e_j$ are disconnected, we obtain the following results by letting $R_i \rightarrow \infty$ or $R_j \rightarrow \infty$ in the formulas given in \lemref{lem resistance one edge} and \thmref{thm resistance pxy}.
\begin{enumerate}
\item $r(x,y)= |x-y|$, if both $x$ and $y$ are on the same edge that is a bridge.
\item
\begin{equation*}\label{eqn res1}
r(p,x)=\begin{cases}x+r(p,\pp), &  \, \text{if $p$ is on the side of $\pp$}, \\
\li-x+r(p,\qq), &  \, \text{if $p$ is on the side of $\qq$},
\end{cases}
\end{equation*}
%\item $r(p,x)= x+r(p,\pp)$, for any vertex $p$ if $e_i$ is a bridge.
\item If both $e_i$ and $e_j$ are bridges that are distinct edges, then we have
\begin{equation*}\label{eqn res2}
r(x,y)=\begin{cases}x+y+r(\pp,p_j), &  \, \text{if $\pp$ and $p_j$ are  between $x$ and $y$}, \\
x+L_j-y+r(\pp,q_j), &  \, \text{if $\pp$ and $q_j$ are  between $x$ and $y$},\\
L_i-x+y+r(q_i,p_j), &  \, \text{if $q_i$ and $p_j$ are  between $x$ and $y$},\\
L_i-x+L_j-y+r(\qq,q_j), &  \, \text{if $\qq$ and $q_j$ are between $x$ and $y$}.
\end{cases}
\end{equation*}
%$r(x,y)=x+y+r(\pp,p_j)$, if both $e_i$ and $e_j$ are bridges that are distinct edges.
\item Suppose only $e_i$ is a bridge (the case that only $e_j$ is a bridge can be done by imitating this case). Then
we have two cases:

If $y$ is on the side of $\pp$, we have
$$r(x,y)=x-y^2 \frac{L_j - r(p_j,q_j)}{L_j^2}+y \frac{L_j-r(p_j,q_j)+r(\pp,q_j)-r(\pp,p_j)}{L_j}+r(\pp,p_j).$$

If $y$ is on the side of $\qq$, we have
$$r(x,y)=\li-x-y^2 \frac{L_j - r(p_j,q_j)}{L_j^2}+y \frac{L_j-r(p_j,q_j)+r(\qq,q_j)-r(\qq,p_j)}{L_j}+r(\qq,p_j).$$
%\begin{equation*}\label{eqn res3}
%r(x,y)=\begin{cases}
%x+r(\pp,y)=x-y^2 \frac{L_j - r(p_j,q_j)}{L_j^2}+y \frac{L_j-r(p_j,q_j)+r(\pp,q_j)-r(\pp,p_j)}{L_j}+r(\pp,p_j), &  \, \text{if $y$ is on %the side of $\pp$}, \\
%\li-x+r(\qq,y)=x-y^2 \frac{L_j - r(p_j,q_j)}{L_j^2}+y \frac{L_j-r(p_j,q_j)+r(\qq,q_j)-r(\qq,p_j)}{L_j}+r(\qq,p_j), &  \, \text{if $y$ is %on the side of $\qq$},
%\end{cases}
%\end{equation*}
%%$r(x,y)=x+r(\pp,y)=x-y^2 \frac{L_j - r(p_j,q_j)}{L_j^2}+y \frac{L_j-r(p_j,q_j)+r(\pp,q_j)-r(\pp,p_j)}{L_j}+r(\pp,p_j)$,
%%\item $r(x,y)=....$, if only $e_j$ is a bridge.
\end{enumerate}
\end{remark}
\begin{remark}\label{rem voltage xy}
Since $2j_x(y,z)=r(x,y)+r(x,z)-r(y,z)$ for any $x, \, y, \, z \in \ga$, one can use \thmref{thm resistance pxy} for each of
$r(x,y)$, $r(x,z)$ and $r(y,z)$ to express the voltage function $j_x(y,z)$ in terms of its values on vertices of $\ga$.
\end{remark}

\section{Arakelov-Green Function $g_{\mu_{can}}(x,y)$ }\label{sec arakelov green function}

In this section, we first give the definition of Arakelov-Green functions $g_\mu(x,y)$ on a metrized graph $\ga$. Then
we study the Arakelov-Green function $g_{\mu_{can}}(x,y)$ defined with respect to a canonical measure $\mu_{can}$ $\ga$.
Our goal is to clarify the combinatorial interpretation of  $g_{\mu_{can}}(x,y)$.

For any real-valued, signed Borel measure $\mu$ on $\Gamma$ with
$\mu(\Gamma)=1$ and $|\mu|(\Gamma) < \infty$, define the function
%\begin{equation*}
$j_{\mu}(x,y) \ = \ \int_{\Gamma} j_{z}(x,y) \, d\mu({z}).$
%\end{equation*}
Clearly $j_{\mu}(x,y)$ is symmetric, and is jointly continuous in
$x$ and $y$. Chinburg and Rumely \cite{CR} discovered that there is a unique real-valued, signed Borel measure $\mu=\mu_{can}$
such that $j_{\mu}(x,x)$ is constant on $\ga$. The measure $\mu_\can$ is called the
%{\it canonical measure}
\textit{canonical measure}. One can find several interpretations of $\mu_{can}$ in the articles \cite{BRh} and \cite{C2}.
Baker and Rumely \cite[Section 14]{BRh} called the constant $\frac{1}{2}j_{\mu}(x,x)$ the \textit{tau constant} of $\ga$
and denoted it by $\tg$. The following lemma gives a description of the tau constant.
\begin{lemma}\cite[Lemma 14.4]{BRh}\label{lemtauformula}
For any fixed $y$ in $\ga$,
$\tg =\frac{1}{4}\int_{\ga}\big(\frac{\partial}{\partial x} r(x,y) \big)^2dx$.
\end{lemma}
One can find more detailed information on $\tg$ in articles \cite{C1}, \cite{C3}, \cite{C4} and \cite{C5}.

Let $\mu$ be a real-valued signed Borel measure of total mass $1$ on
$\Gamma$. In the article \cite{BRh}, the \textit{Arakelov-Green's function} $g_\mu(x,y)$ associated to
$\mu$ is defined to be
$$g_\mu(x,y) = \int_\Gamma j_z(x,y) \, d\mu(z) - \int_{\Gamma^3}
j_z(x,y) \, d\mu(z) d\mu(x) d\mu(y),$$
where the latter integral is a constant that depends on $\ga$ and $\mu$.

As shown in the article \cite{BRh}, $g_{\mu}(x,y)$ is continuous,
symmetric (i.e., $g_{\mu}(x,y)=g_{\mu}(y,x)$, for each $x$ and $y$),
and for each $y$, $\int_{\Gamma} g_{\mu}(x,y) \, d\mu(x) \ = \ 0 \ $.
More precisely, as shown in the article \cite{BRh}, one can characterize
$g_\mu(x,y)$ as the unique function on $\Gamma \times \Gamma$ such
that
\begin{itemize}
\item[$(1)$]
% (Smoothness)
$g_\mu(x,y)$ is jointly continuous in $x, y$ and belongs to
$\BDV_\mu(\Gamma)$ as a
 function of $x$, for each fixed $y$, where $\BDV_\mu(\Gamma) := \{ f \in \BDV(\Gamma) \; : \; \int_\Gamma f \,
d\mu = 0 \}$ and $BDV(\Gamma)$ is space of continuous functions of bounded differential
variation $\Gamma$.
\item[$(2)$]
% (Differential equation)
For fixed $y$, $g_\mu$ satisfies the identity
$
\Delta_x g_\mu(x,y) = \delta_y(x) - \mu(x).
$
\item[$(3)$]
% (Normalization)
$\iint_{\Gamma \times \Gamma} g_\mu(x,y) d\mu(x) d\mu(y) = 0$.
\end{itemize}
Precise definitions of $\BDV(\Gamma)$, and of $\Delta
f$ for $f \in \BDV(\Gamma)$, can be found in \cite{BRh}.
%Rumely and Chinburg in \cite{CR} discovered the ``Canonical Measure'' $\mu_{can}$, a measure of total mass $1$, on a given %metrized graph $\ga$.
%See \cite{BRh} and \cite{C1} for several interpretations of $\mu_{can}$.
%
%An important observation is that the Laplacian exists as a
%measure-valued operator on a larger space than $\Zh(\Gamma)$. A
%space $\BDV(\Gamma) \subset \cC(\Gamma)$, which is in a precise
%sense the largest space of continuous functions $f$ for which
%$\Delta f$ exists as a complex Borel measure, is defined in
%\cite{BRh}. (The abbreviation BDV stands for ``bounded differential
%variation''. Precise definitions of $\BDV(\Gamma)$, and of $\Delta
%f$ for $f \in \BDV(\Gamma)$, can be found in \cite{BRh}.)
%

Arakelov-Green function $g_\mu(x,y)$ satisfies the following properties (a detailed proof can be found in \cite[Theorem 2.11]{CR} and \cite[pg. 34]{C1}):
\begin{theorem}\cite[Theorem 14.1]{BRh}\label{thm arakelov-green canonical}
\begin{enumerate}
\item The probability measure $\mu_{can}=\Delta_x(\frac{1}{2}r(x,y))+\delta_y(x)$ is independent of $y \in \ga$.
\item $\mu_{can}$ is the unique measure $\mu$ of total mass $1$ on $\ga$ for which $g_\mu(x,x)$ is a constant independent of $x$.
\item There is a constant $\tg \in \RR$ such that $g_{\mu_{can}}(x,y)=-\frac{1}{2}r(x,y)+\tg$.
\end{enumerate}
\end{theorem}

Since $r(x,x)=0$ for every $x \in \ga$, the diagonal values $g_{\mu_{can}}(x,x)$ are constant on $\ga$, and are equal to the tau constant $\tg$.

Next, we state the main result of this paper:
\begin{theorem}\label{thm main1}
Suppose $e_i \in \ee{\ga}$ is an edge of length $\li$ with end points $p_i$ and $q_i$, and $e_j \in \ee{\ga}$ is an edge of length $L_j$ with end points $p_j$ and $q_j$. Assume that the edges $e_i$ and $e_j$ are not bridges and distinct edges, but their end points are not necessarily distinct. If $x$ belongs to the edge $e_i$ and $y$ belongs to the edge $e_j$, we have
\begin{equation*}\label{eqn arakelov green function formula1}
\begin{split}
g_{\mu_{can}}(x,y) &= \tg +x^2 \frac{\li - r(\pp,\qq)}{2\li^2}+y^2 \frac{L_j - r(p_j,q_j)}{2L_j^2}
 -\frac{xy}{\li L_j} \big(j_{p_j}(\pp,q_j)-j_{p_j}(q_i,q_j)\big)\\
& \qquad -\frac{x}{2\li} \big(\li-2j_{p_i}(\qq,p_j)\big)
 -\frac{y}{2L_j} \big(L_j-2j_{p_j}(\pp,q_j)\big)
 -\frac{1}{2}r(\pp,p_j).
\end{split}
\end{equation*}
If both $x$ and $y$ belong to the same edge $e_i$ of length $\li$ with end points $p_i$ and $q_i$, then we have
\begin{equation*}\label{eqn arakelov green function formula2}
\begin{split}
g_{\mu_{can}}(x,y) = \tg -  \frac{1}{2}|x-y|+(x-y)^2 \frac{\li - r(\pp,\qq)}{2\li^2}.
\end{split}
\end{equation*}
\end{theorem}
\begin{proof}
The result follows from \thmref{thm arakelov-green canonical} along with \lemref{lem resistance one edge} and \thmref{thm resistance pxy}.
\end{proof}
If any of the involved edges in \thmref{thm main1} is a bridge, then we interpret the given formulas by using \remref{rem resistance xy on bridge}. In such cases, we obtain the following modified version of \thmref{thm main1} by applying \thmref{thm arakelov-green canonical} and \remref{rem resistance xy on bridge}:
\begin{theorem}\label{thm main2}
Suppose $e_i \in \ee{\ga}$ is an edge of length $\li$ with end points $p_i$ and $q_i$, and $e_j \in \ee{\ga}$ is an edge of length $L_j$ with end points $p_j$ and $q_j$. Let $x$ and $y$ belong to $e_i$ and $e_j$, respectively.
\begin{enumerate}
\item If both $x$ and $y$ are on the same edge that is a bridge, we have
$$g_{\mu_{can}}(x,y)= \tg- \frac{1}{2}|x-y|.$$
\item If both $e_i$ and $e_j$ are bridges that are distinct edges, then we have
\begin{equation*}\label{eqn res2b}
g_{\mu_{can}}(x,y)=\begin{cases} \tg -\frac{1}{2}\big(x+y+r(\pp,p_j)\big), &  \, \text{if $\pp$ and $p_j$ are  between $x$ and $y$}, \\
\tg- \frac{1}{2}\big(x+L_j-y+r(\pp,q_j) \big), &  \, \text{if $\pp$ and $q_j$ are  between $x$ and $y$},\\
\tg- \frac{1}{2}\big(L_i-x+y+r(p_j,q_i)\big), &  \, \text{if $p_j$ and $q_i$ are  between $x$ and $y$},\\
\tg- \frac{1}{2}\big(L_i-x+L_j-y+r(\qq,q_j)\big), &  \, \text{if $\qq$ and $q_j$ are between $x$ and $y$}.
\end{cases}
\end{equation*}
\item Suppose only $e_i$ is a bridge (the case that only $e_j$ is a bridge can be done similar to this case). Then
we have two cases:

If $y$ is on the side of $\pp$, we have
\begin{equation*}\label{eqn gmucana}
\begin{split}
g_{\mu_{can}}(x,y)& =\tg+y^2 \frac{L_j - r(p_j,q_j)}{2L_j^2}-y \frac{L_j-r(p_j,q_j)+r(\pp,q_j)-r(\pp,p_j)}{2L_j}\\
& \qquad -\frac{1}{2}\big(x +r(\pp,p_j) \big).
\end{split}
\end{equation*}
%$$g_{\mu_{can}}(x,y)=\tg- \frac{1}{2}\Big(x-y^2 \frac{L_j - r(p_j,q_j)}{L_j^2}+y %\frac{L_j-r(p_j,q_j)+r(\pp,q_j)-r(\pp,p_j)}{L_j}+r(\pp,p_j) \Big).$$

If $y$ is on the side of $\qq$, we have
\begin{equation*}\label{eqn gmucanb}
\begin{split}
g_{\mu_{can}}(x,y)& =\tg +y^2 \frac{L_j - r(p_j,q_j)}{2L_j^2}-y \frac{L_j-r(p_j,q_j)+r(\qq,q_j)-r(\qq,p_j)}{2L_j}\\
& \qquad -\frac{1}{2}\big(\li-x +r(\qq,p_j) \big).
\end{split}
\end{equation*}
%$$g_{\mu_{can}}(x,y)=\tg- \frac{1}{2}\Big(\li-x-y^2 \frac{L_j - r(p_j,q_j)}{L_j^2}+y %\frac{L_j-r(p_j,q_j)+r(\qq,q_j)-r(\qq,p_j)}{L_j}+r(\qq,p_j) \Big).$$
\end{enumerate}
\end{theorem}

Recall that $g_{\mu_{can}}(x,y)$ is a symmetric function, i.e., $g_{\mu_{can}}(x,y)=g_{\mu_{can}}(y,x)$, and that it is continuous in $x$ and $y$. It is clear from \thmref{thm main1} that $g_{\mu_{can}}(x,y)$ is a piece-wisely defined function on each pair of edges $(e_i, e_j)$. Based on these information about $g_{\mu_{can}}(x,y)$ and \thmref{thm main1}, we suggest that a matrix $\zm$ defined below can be used to
describe $g_{\mu_{can}}(x,y)$. We call $\zm$ the value matrix of $g_{\mu_{can}}(x,y)$.

We define $\zm=(z_{ij})$ as a matrix of size $e \times e$, where $e$ is the number of edges of $\ga$, such that $z_{ij}$ is equal to $g_{\mu_{can}}(x,y)$ when $x \in e_i$ and $y \in e_j$. We note that $\zm$ is a symmetric matrix. The diagonal values of $\zm$ are of the form $a_2(x-y)^2+a_1 (x-y)+a_0$ for some constants $a_0$, $a_1$ and $a_2$, where $a_2 =0$ iff the edge that $x$ and $y$ belong to is a bridge. Other entries of $\zm$ are of the form
$a x^2 + b y^2 + c x y+ d x+e y+f$ for some constants $a$, $b$, $c$, $d$, $e$ and $f$, where $e_i$ is a bridge iff $a=0$ and that $e_j$ is a bridge iff $b=0$. We provide various examples in \secref{sec computations}.

\section{ Arakelov-Green Function $g_{\mu_{D}}(x,y)$ }\label{sec arakelov green function2}
In this section, we consider another important Arakelov-Green function $g_{\mu_{D}}(x,y)$ defined by Zhang \cite[Section 3]{Zh1} as the generalization of $g_{\mu_{can}}(x,y)$. Here, $g_{\mu_{D}}(x,y)$ is defined with respect to the measure $\mu_D(x)$, where $D$ is a measure on $\ga$. More precisely,
for any divisor $D=\sum_{q \in \vv{\ga}}a_q \cdot q$ on $\ga$ with $\deg(D) \neq -2$ and for the corresponding measure (called admissible metric on $\ga$ with respect to $D$)
$$\mu_D(x)=\frac{1}{\deg(D)+2}(\sum_{q \in \vv{\ga}}a_q \dd{q}(x)+2 \mu_{can}(x)),$$
$g_{\mu_{D}}(x,y)$ can be given as follows \cite[Section 4.4]{C1}:
\begin{equation}\label{eqn arakelov green zh}
\begin{split}
g_{\mu_{D}}(x,y)=\frac{1}{\deg(D)+2}
\Big(\sum_{s \in \vv{\ga}}a_s \cdot j_s(x,y) + 4 \tg -r(x,y) \Big) - \cd,
\end{split}
\end{equation}
where
\begin{equation*}
\begin{split}
\cd = \frac{1}{2(\deg(D)+2)^2}\Big( 8 \tg (\deg(D)+1) + \sum_{q, \, s \, \in \vv{\ga}} a_q \cdot a_s \cdot r(q,s) \Big).
\end{split}
\end{equation*}
Note that $g_{\mu_{D}}(x,y)=g_{\mu_{can}}(x,y)$ and $\mu_D(x)=\mu_{can}$ if $D=0$.

Using \thmref{thm resistance pxy}, \lemref{lem resistance one edge}, \remref{rem voltage xy} and \eqnref{eqn arakelov green zh}, one can extend \thmref{thm main1} and \thmref{thm main2} to a formula for $g_{\mu_{D}}(x,y)$.

\section{Computational Examples For $g_{\mu_{can}}(x,y)$ }\label{sec computations}

We first give two examples for symbolic computations, and then an example with numerical computations.
Given a metrized graph $\ga$ with discrete Laplacian $\lm$, we first compute the pseudo inverse $\plm$ of $\lm$. We can compute the tau constant $\tg$ symbolically for certain graphs or numerically for all graphs by using $\lm$ and $\plm$ as shown in \cite[Theorem 1.1]{C3}. Then we compute the resistance matrix $\rmm$ using the matrix $\plm$ along with \lemref{lem resistance and voltage} given below. Finally, we compute the value matrix $\zm$ using either \thmref{thm main1} or \thmref{thm main2}.

We first recall that both voltage and resistance values on vertices can be expressed in terms of the entries of pseudo inverse $\plm$ of $\lm$ (see \cite[Lemmas 3.4 and 3.5]{C2} and the related references therein):
\begin{lemma}\label{lem resistance and voltage}
For any $p$, $q$, $s$ in $\vv{\ga}$, we have
\begin{equation*}\label{}
\begin{split}
r(p,q)=l_{pp}^+-2l_{pq}^+ + l_{qq}^+, \qquad \text{and} \quad
j_p(q,s)=l_{pp}^+ - l_{pq}^+ -l_{ps}^+ +l_{qs}^+.
\end{split}
\end{equation*}
\end{lemma}

Suppose that $J_e$ denote an $e \times e$ matrix with each entry is equal to $1$.

\textbf{Example I:}

Let $\ga$ be the circle graph with three vertices as illustrated in \figref{fig circlegraph}. The total length of $\ga$ is $\elg=a+b+c$.
We have $\tg=\frac{\elg}{12}$, and the following discrete Laplacian matrix $\lm$, its pseudo inverse $\plm$, the resistance matrix $\rmm$ and the value matrix $\zm$ with respect to the ordered end points of edges $\{ (v_1, v_2), (v_1, v_3), (v_2, v_3) \}$:
\begin{figure}
\centering
\includegraphics[scale=0.75]{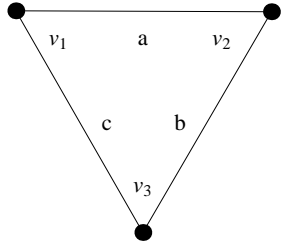} \caption{A circle graph with vertices $\{ v_1, \, v_2, \, v_3 \}$,
and edge lengths $\{ a, \, b, \, c \}$.} \label{fig circlegraph}
\end{figure}
\tiny
%\[
%\plm=\left[
%\begin{array}{ccc}
% \frac{1}{a}+\frac{1}{b} & -\frac{1}{a} & -\frac{1}{b} \\
% -\frac{1}{a} & \frac{1}{a}+\frac{1}{c} & -\frac{1}{c} \\
% -\frac{1}{b} & -\frac{1}{c} & \frac{1}{b}+\frac{1}{c} \\
%\end{array}
%\right]
%\]
\[
\lm=\left[
\begin{array}{ccc}
 1/a+1/b & -1/a & -1/b \\
 -1/a & 1/a+1/c & -1/c \\
 -1/b & -1/c & 1/b+1/c \\
\end{array}
\right],
\qquad
\rmm=\frac{1}{\elg}
\left[
\begin{array}{ccc}
 0 & a b+a c & a b+b c \\
 a b+a c & 0 & a c+b c \\
 a b+b c & a c+b c & 0 \\
\end{array}
\right],
\]
\[
\plm=\frac{1}{9 \elg}
\left[
\begin{array}{ccc}
 b c+a (4 b+c) & b c-2 a (b+c) & -2 b c+a (-2 b+c) \\
 b c-2 a (b+c) & b c+a (b+4 c) & a (b-2 c)-2 b c \\
 -2 b c+a (-2 b+c) & a (b-2 c)-2 b c & 4 b c+a (b+c) \\
\end{array}
\right],
\]
\[
\zm=\tg J_3 -\frac{1}{2 \elg}
\left[
\begin{array}{ccc}
 -(x-y)^2+(a+b+c) |x-y| & (a+b+c-x-y) (x+y) & (b+c+x-y) (a-x+y) \\
 (a+b+c-x-y) (x+y) & -(x-y)^2+(a+b+c) |x-y| & (b+c-x-y) (a+x+y) \\
 (b+c+x-y) (a-x+y) & (b+c-x-y) (a+x+y) & -(x-y)^2+(a+b+c) |x-y| \\
\end{array}
\right].
\]
\normalsize
Since $\ga$ has no bridge, each entry of $\zm$ is a quadratic function in both $x$ and $y$.

\textbf{Example II:}

Let $\ga$ be the tree graph as given in \figref{fig treegraph}. The list of the ordered end points of the edges is
$\{(v_1, v_3), (v_2, v_3 ), (v_3, v_4), (v_4, v_5), (v_4, v_6) \}$, and the list of the corresponding edge lengths in order is given by $\{ a, \, b, \, c, \, d, \, e \}$.
\begin{figure}
\centering
\includegraphics[scale=0.75]{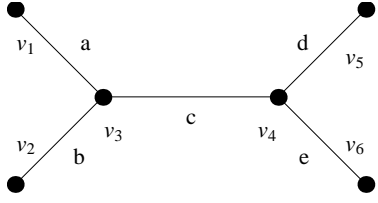} \caption{A tree graph with vertices $\{ v_1, \, v_2, \, v_3, \, v_4, \, v_5, \, v_6 \}$,
and edge lengths $\{ a, \, b, \, c, \, d, \, e \}$.} \label{fig treegraph}
\end{figure}
In this case, $\tg=\frac{1}{4}(a+b+c+d+e)$, and the Laplacian matrix $\lm$ and the resistance matrix $\rmm$ are given as follows:
\tiny
\[
\lm=
\left[
\begin{array}{cccccc}
 1/a &  & 0 & -1/a & 0 & 0 \\
 0 & 1/b & -1/b & 0 & 0 & 0\\
 -1/a & -1/b & 1/a+1/b+1/c & -1/c & 0 & 0\\
 0 & 0 & -1/c & 1/c+1/d+1/e & -1/d & -1/e\\
 0 & 0 & 0 & -1/d & 1/d & 0\\
 0 & 0 & 0 & -1/e & 0&1/e\\
\end{array}
\right],
\]
%\[
%\plm=\frac{1}{36}\left[
%\begin{array}{cccccc}
% 25 a+b+9 c+d+e & -5 a-5 b+9 c+d+e & -5 a+b+9 c+d+e & -5 a+b-9 c+d+e & -5 a+b-9 c-5 d+e & -5 a+b-9 c+d-5 e \\
% -5 a-5 b+9 c+d+e & a+25 b+9 c+d+e & a-5 b+9 c+d+e & a-5 b-9 c+d+e & a-5 b-9 c-5 d+e & a-5 b-9 c+d-5 e \\
% -5 a+b+9 c+d+e & a-5 b+9 c+d+e & a+b+9 c+d+e & a+b-9 c+d+e & a+b-9 c-5 d+e & a+b-9 c+d-5 e \\
% -5 a+b-9 c+d+e & a-5 b-9 c+d+e & a+b-9 c+d+e & a+b+9 c+d+e & a+b+9 c-5 d+e & a+b+9 c+d-5 e \\
% -5 a+b-9 c-5 d+e & a-5 b-9 c-5 d+e & a+b-9 c-5 d+e & a+b+9 c-5 d+e & a+b+9 c+25 d+e & a+b+9 c-5 d-5 e \\
% -5 a+b-9 c+d-5 e & a-5 b-9 c+d-5 e & a+b-9 c+d-5 e & a+b+9 c+d-5 e & a+b+9 c-5 d-5 e & a+b+9 c+d+25 e \\
%\end{array}
%\right]
%\]
\[
\rmm=
\left[
\begin{array}{cccccc}
 0 & a+b & a & a+c & a+c+d & a+c+e \\
 a+b & 0 & b & b+c & b+c+d & b+c+e\\
 a & b & 0 & c & c+d & c+e\\
 a+c & b+c & c & 0 & d & e\\
 a+c+d & b+c+d & c+d & d & 0 & d+e\\
 a+c+e & b+c+e & c+e & e & d+e&0\\
\end{array}
\right],
\]
%\[
%Z=
%\left(
%\begin{array}{ccccc}
% \tg-\frac{1}{2}|x-y| & \tg-\frac{1}{2}(a-x+b-y) & \tg-\frac{1}{2}(a-x+y) & \tg-\frac{1}{2}(a-x+c+y) & \tg-\frac{1}{2}(a-x+c+y)\\
%\tg-\frac{1}{2}(a-x+b-y) & \tg-\frac{1}{2}|x-y| & \tg-\frac{1}{2}(b-x+y) & \tg-\frac{1}{2}(b-x+c+y) & \tg-\frac{1}{2}(b-x+c+y)\\
%\tg-\frac{1}{2}(a-x+y) & \tg-\frac{1}{2}(b-x+y) & \tg-\frac{1}{2}|x-y| & \tg-\frac{1}{2}(c-x+y) & \tg-\frac{1}{2}(c-x+y)\\
%\tg-\frac{1}{2}(a-x+c+y) &  \tg-\frac{1}{2}(b-x+c+y) & \tg-\frac{1}{2}(c-x+y) & \tg-\frac{1}{2}|x-y| & \tg-\frac{1}{2}(x+y)\\
% \tg-\frac{1}{2}(a-x+c+y)& \tg-\frac{1}{2}(b-x+c+y) &\tg-\frac{1}{2}(c-x+y) & \tg-\frac{1}{2}(x+y) & \tg-\frac{1}{2}|x-y|
%\end{array}
%\right)
%\]
\[
\zm=\tg J_5 - \frac{1}{2}
\left[
\begin{array}{ccccc}
|x-y| & a-x+b-y & a-x+y & a-x+c+y & a-x+c+y\\
a-x+b-y & |x-y| & b-x+y & b-x+c+y & b-x+c+y\\
a-x+y & b-x+y & |x-y| & c-x+y & c-x+y\\
a-x+c+y & b-x+c+y & c-x+y & |x-y| & x+y\\
a-x+c+y& b-x+c+y &c-x+y & x+y & |x-y|
\end{array}
\right].
\]
\normalsize
%Here $\zm$ is given with respect to the ordered end points of edges $\{(v_1, v_3), (v_2, v_3), (v_3, v_4), (v_4, v_5), (v_4, v_6)\}$.
We considers the following cases to clarify how we use the value matrix $\zm$.

If $x, \, y \in e_1$, $g_{\mu_{can}}(x,y)=\tg-\frac{1}{2}|x-y|$, where $0 \leq x \leq a$, $0 \leq y \leq a$ and $v_1$ corresponds to $0$.

If $x, \, y \in e_3$, $g_{\mu_{can}}(x,y)=\tg-\frac{1}{2}|x-y|$, where $0 \leq x \leq c$, $0 \leq y \leq c$ and $v_3$ corresponds to $0$.

If $x \in e_2$ and $y \in e_4$, $g_{\mu_{can}}(x,y)=\tg-\frac{1}{2}(b-x+c+y)$, where $0 \leq x \leq b$, $0 \leq y \leq d$ and both $v_2$ and $v_4$ correspond to $0$.

Note that each entry of $\zm$ is a linear function in both $x$ and $y$ because of the fact that $\ga$ is a tree, i.e., has no bridges.

\textbf{Example III:}

In this example, we consider the tetrahedral graph with edge lengths given as in \figref{fig agf2}.

\begin{figure}
\centering
\includegraphics[scale=0.5]{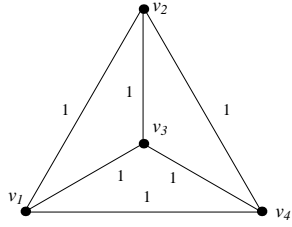} \caption{Tetrahedral graph with vertices $\{v_1, \, v_2, \, v_3, \, v_4 \}$.} \label{fig agf2}
\end{figure}

In this case, we have $\tg=\frac{5}{16}$. The discrete Laplacian matrix $\lm$, its pseudo inverse of $\plm$ and the resistance matrix $\rmm$ are as follows:
\tiny
\[
\lm =
\left[
\begin{array}{cccc}
 3 & -1 & -1 & -1 \\
 -1 & 3 & -1 & -1 \\
 -1 & -1 & 3 & -1 \\
 -1 & -1 & -1 & 3 \\
\end{array}
\right],
\qquad
\plm=
\frac{1}{48}
\left[
\begin{array}{cccc}
 19 & 7 & 7 & 7 \\
 7 & 19 & 7 & 7 \\
 7 & 7 & 19 & 7 \\
 7 & 7 & 7 & 19 \\
\end{array}
\right],
\qquad
\rmm=\frac{1}{2}
\left[
\begin{array}{cccc}
 0 & 1 & 1 & 1 \\
 1 & 0 & 1 & 1 \\
 1 & 1 & 0 & 1 \\
 1 & 1 & 1 & 0 \\
\end{array}
\right].
\]
\normalsize
If the ordered end points of edges is $\{(v_1, v_2), (v_1, v_3), (v_1, v_4), (v_2, v_3), (v_2, v_4), (v_3, v_4)\}$,
we have the value matrix $\zm=\frac{5}{16} J_6-\frac{1}{4}[C_1, \, C_2, \, C_3, \, C_4, \, C_5, \, C_6]$, where $C_i$ with $i \in \{1, \, 2, \, 3, \, 4, \, 5, \, 6  \}$ are columns as given below:
\tiny
\[
[C_1,\, C_2, \, C_3]=
\left[
\begin{array}{ccc}
 -x^2+2 x y-y^2+2 |x-y| & 2 x-x^2+2 y-x y-y^2 & 2 x-x^2+2 y-x y-y^2 \\
 2 x-x^2+2 y-x y-y^2 & -x^2+2 x y-y^2+2 |x-y| & 2 x-x^2+2 y-x y-y^2 \\
 2 x-x^2+2 y-x y-y^2 & 2 x-x^2+2 y-x y-y^2 & -x^2+2 x y-y^2+2 |x-y| \\
 1-x^2+y+x y-y^2 & 1+x-x^2+y-x y-y^2 & 1+x-x^2+y-y^2 \\
 1-x^2+y+x y-y^2 & 1+x-x^2+y-y^2 & 1+x-x^2+y-x y-y^2 \\
 1+x-x^2+y-y^2 & 1-x^2+y+x y-y^2 & 1+x-x^2+y-x y-y^2 \\
\end{array}
\right],
\]
\[
[C_4,\, C_5, \, C_6]=
\left[
\begin{array}{ccc}
 1-x^2+y+x y-y^2 & 1-x^2+y+x y-y^2 & 1+x-x^2+y-y^2 \\
 1+x-x^2+y-x y-y^2 & 1+x-x^2+y-y^2 & 1-x^2+y+x y-y^2 \\
 1+x-x^2+y-y^2 & 1+x-x^2+y-x y-y^2 & 1+x-x^2+y-x y-y^2 \\
 -x^2+2 x y-y^2+2 |x-y| & 2 x-x^2+2 y-x y-y^2 & 1-x^2+y+x y-y^2 \\
 2 x-x^2+2 y-x y-y^2 & -x^2+2 x y-y^2+2 |x-y| & 1+x-x^2+y-x y-y^2 \\
 1-x^2+y+x y-y^2 & 1+x-x^2+y-x y-y^2 & -x^2+2 x y-y^2+2 |x-y| \\
\end{array}
\right].
\]
\normalsize
Note that each entry of $\zm$ is a quadratic function in both $x$ and $y$ as expected, because $\ga$ has no bridge.

%Here, put something that won't work so that it will generate an error and we will see the erronous parts.

%$\cite{C1}$

\textbf{Acknowledgements:} This work is supported by The Scientific and Technological Research Council of Turkey-TUBITAK (Project No: 110T686).


\begin{thebibliography}{999}

\bibitem[1]{A} S. Yu. Arakelov, Intersection theory of divisors on an arithmetic surface, {\em Math. USSR
Izvestiya} 8 (1974), 1167–1180. (English translation.)

\bibitem[2]{BF} M. Baker and X. Faber, Metrized graphs, Laplacian operators, and
electrical networks, {\em Quantum graphs and their applications}, 15--33,
Contemp. Math., 415, Amer. Math. Soc., Providence, RI, (2006).

\bibitem[3]{BRh} M. Baker and R. Rumely, Harmonic analysis on metrized graphs, {\em Canadian
J. Math}: 59, no. 2, 225--275, (2007).
%\\http://arxiv.org/abs/math.NT/0407427


\bibitem[4]{CR} T. Chinburg and R. Rumely, The capacity pairing,
{\em J. reine angew. Math.}  434, 1--44, (1993).

\bibitem[5]{C1} Z. Cinkir, {\em The Tau Constant of Metrized Graphs},
Thesis at University of Georgia, (2007).

\bibitem[6]{C2}Z. Cinkir, Generalized Foster's identities, {\em International Journal of Quantum Chemistry}, Volume 111, Issue 10, (2011), 2228–-2233.
%
%\\DOI: 10.1002/qua.22521
%Article first published online: 2 MAR 2010. DOI: 10.1002/qua.22521
%The older version can be found at \\http://arxiv.org/abs/0907.3770v2

\bibitem[7]{C3} Z. Cinkir, The tau constant and the discrete Laplacian matrix of a metrized graph, {\em European Journal of Combinatorics}, Volume 32, Issue 4, (2011), 639--655.

\bibitem[8]{C4} Z. Cinkir, The tau constant of a metrized graph and its behavior under graph operations, {\em The Electronic Journal of Combinatorics}, Volume 18 (1) (2011) P81.

\bibitem[9]{C5} Z. Cinkir, The tau constant and the edge connectivity of a metrized graph,
{\em The Electronic Journal of Combinatorics}, Volume 19, Issue 4, (2012) P46.

%\bibitem[10]{C6} Z. Cinkir, Deletion and contraction identities for the resistance values and the Kirchhoff index, {\em International %Journal of Quantum Chemistry}, Volume:111, Issue:15 (2011) 4030--4041.


\bibitem[10]{C7} Z. Cinkir, Zhang's Conjecture and the Effective Bogomolov Conjecture over function fields, {\em Inventiones Mathematicae}, Volume 183, Number 3, (2011) 517--562.

%\bibitem[10]{Fa} X. W. C. Faber, The geometric bogomolov conjecture for curves of small genus,
%{\em Experiment. Math.}, 18(3):347--367, (2009).
%%preprint, c.f. at \, http://arxiv.org/abs/0803.0855v2

\bibitem[11]{F} G. Faltings,   Calculus on arithmetic surfaces, {\em Ann. Math.} 119 (1984), 387–424.

\bibitem[12]{RumelyBook} R. Rumely, {\em Capacity Theory on Algebraic Curves}, Lecture Notes in
Mathematics 1378, Springer-Verlag, Berlin-Heidelberg-New York, (1989).


% \`{o} produces a grave accent, ò
% \'{o} produces an acute accent, ó
% \^{o} produces a circumflex, ô
% \"{o} produces an umlaut or dieresis, ö

%\bibitem[27]{KY1} K. Yamaki, Effective calculation of the geometric height and the Bogomolov
%conjecture for hyperelliptic curves over function fields,
%{\em J. Math. Kyoto Univ.}, 48-2:401--443, (2008).
%%{\em Bogomolov's  conjectures for hyperelliptic curves over function fields},
%%preprint. Available at \\http://arxiv.org/abs/math/9903066

%\bibitem[28]{KY2} K. Yamaki, Geometric Bogomolov's conjecture for curves of genus $3$ over function fields,
%{\em J. Math. Kyoto Univ.}, 42-1:57--81, (2002).

\bibitem[13]{Zh1} S. Zhang, Admissible pairing on a curve,
{\em Invent. Math.} Volume 112, (1993) 171--193.

\bibitem[14]{Zh2} S. Zhang, Gross--Schoen cycles and dualising sheaves, {\em Invent. Math.},
Volume 179, Number 1, (2010) 1--73.

%\bibitem[30]{Zh2} S. Zhang, Gross--Schoen cycles and dualising sheaves, {\em Invent. Math.},
%published online: 9 July 2009.
%%\\ http://www.math.columbia.edu/$\sim$szhang/papers/Preprints.htm


\end{thebibliography}
\end{document}